\documentclass[12pt]{article}
\usepackage{amssymb,amsmath,amsthm,bm,mathtools,faktor,microtype}
\usepackage[total={148mm,208mm}]{geometry}  
\usepackage[bitstream-charter]{mathdesign}
\usepackage[english]{babel}
\usepackage[utf8x]{inputenc}
\usepackage[colorinlistoftodos]{todonotes}
\usepackage[colorlinks=true, allcolors=blue,hyperindex,breaklinks]{hyperref}
\date{}

\author{Johan Andersson\thanks{Email:johan.andersson@oru.se \, Address:Division of Mathematics, School of Science and Technology, {\"O}rebro University, {\"O}rebro, SE-701 82 Sweden.}}

\title{Discrete universality, continuous universality and hybrid universality are equivalent}

\theoremstyle{plain}
\newtheorem{thm}{Theorem}

\theoremstyle{definition}

\newtheorem{defn}{Definition}

\newtheorem{ack}{Acknowledgements}

\newcommand{\strip}{ S}
\DeclareOldFontCommand{\rm}{\normalfont\rmfamily}{\mathrm}
\DeclareOldFontCommand{\sf}{\normalfont\sffamily}{\mathsf}
\DeclareOldFontCommand{\tt}{\normalfont\ttfamily}{\mathtt}
\DeclareOldFontCommand{\bf}{\normalfont\bfseries}{\mathbf}
\DeclareOldFontCommand{\it}{\normalfont\itshape}{\mathit}
\DeclareOldFontCommand{\sl}{\normalfont\slshape}{\@nomath\sl}
\DeclareOldFontCommand{\sc}{\normalfont\scshape}{\@nomath\sc}
\newcommand{\nicefrac}[2]{\leave§ vmode\kern.1em\raise.5ex\hbox{\the\scriptfont0 #1}\kern-.1em/\kern-.15em\lower.25ex\hbox{\the\scriptfont0 #2}}
\def\halv{\mathchoice{{ \frac 1 2}}{1/2}{1/2}{1/2}}

\newcommand{\C}{{\mathbb C}}

\newcommand{\Q}{{\mathbb Q}} 
\newcommand{\R}{{\mathbb R}}

\newcommand{\abs}[1]{{\left| {#1} \right|}}

\renewcommand{\Re}{\operatorname{Re}} 

\renewcommand{\Im}{\operatorname{Im}}

\begin{document}
\maketitle
\begin{abstract}
Recently Sourmelidis proved that the discrete universality theorem is equivalent to the continuous universality theorem for zeta-functions. He treats both the zero-free universality theorem and the strong universality theorem. Unfortunately in the zero-free case his result  is conditional on a Riemann hypothesis, and in the strong universality case he only proves the implication in one direction. We prove this equivalence unconditionally, and also prove an equivalence with hybrid universality.  While the main application of our result is on Dirichlet series and zeta-functions, our proof is  general and does not use the fact that the functions in question can be represented by   Dirichlet series.
\end{abstract}

\section{Introduction}
In a 1975 paper \cite{Voronin}, Voronin 
proved that any zero-free continuous function $f(s)$ on
 some small disc $K=\{s:|s-3/4|\leq \delta<1\}$, where $0<\delta<1/4$, such that $f$  is analytic in its interior $K^\circ$  may be uniformly approximated on $K$ to    any desired degree of accuracy by the Riemann zeta-function $\zeta(s+it)$.
 Furthermore his method of proof also gives  that the approximation holds with a positive lower density of $t\in [0,T]$ as $T \to \infty$. This result is known as the Voronin universality theorem.
Starting from the late seventies Gonek \cite{Gonek}, Bagchi \cite{Bagchi} and others proved many variants of this theorem. In particular it is a consequence of the Mergelyan theorem that $K$ may be chosen as any compact set $K \subset \{s: 1/2<\Re(s)<1\}$ with connected complement.   Reich \cite{Reich} proved a discrete version of this result: The Riemann zeta function\footnote{Reich proved the result in the context of Dedekind zeta-functions.} $\zeta(s+i\alpha n_k)$ for a subsequence $n_k$ of the integers with positive lower density can uniformly approximate the function $f(s)$ on the set $K$. For certain applications \cite[section 15]{Matsumoto},\cite{KacKul},\cite{NakPan}  it is useful to have a hybrid universality theorem which gives a positive lower bound for $t$  such that the simultaneous approximation $|\zeta(s+it)-f(s)|<\varepsilon$ for $s \in K$ and    $|\lambda_j^{-it}-a_j|<\varepsilon$ for $j=1,\ldots,k$  when $|a_j|=1$ and $\log \lambda_j$  are linearly independent over $\Q$, is satisfied. This version of universality was first introduced by Gonek \cite{Gonek}. 
 Different variants of universality  have been proven for a multitude of zeta and $L$-functions. 
For a good introduction to the field, and the state of the art of the field in 2007, see Steuding's lecture notes \cite{Steuding}. For more recent surveys, see Matsumoto \cite{Matsumoto} from 2015 and Ka\v cinskait\'e-Matsumoto \cite{RomMat} from 2022. When  universality is proven for  some new Dirichlet series or zeta-function, it is often the case that it is first proved in the continuous setting, and that it is later proved in the discrete and/or hybrid setting.
For examples of this see \cite[Section 11, Section 15]{Matsumoto}.  The proof method for discrete and continuous universality
is  similar, but there are certain key steps where changes have to be made.   Whether either the discrete or the continuous universality theorem implies the other has however until recently not been clear.   Steuding remarks \cite[p.110]{Steuding}  about discrete universality that although their method of proof is similar  "these results do not contain the continuous analogues ... nor that they can be deduced from them". In a recent paper Sourmelidis \cite{Sourmelidis} used deep methods of linear dynamics to show that the concepts are in fact "two sides of the same coin". By assuming a Riemann hypothesis, he proved that the discrete and continuous universality theorem are equivalent. Unfortunately the Riemann hypothesis is not known for the Riemann zeta-function, any Dirichlet $L$-function or any element of the Selberg class of degree $d \geq 1$. The only zeta-functions where this result is applicable unconditionally are certain Selberg zeta-functions where the Riemann hypothesis is known and universality has been proved \cite{DruGarKac}, \cite{Mishou2}.
In the case of strong universality, where we remove the zero-free condition he managed to prove that the strong continuous universality theorem implies the discrete counterpart, whereas in the opposite direction he obtains the same result when $\liminf$ is replaced with $\limsup$. The aim of our paper is to prove the results of Sourmelidis unconditionally. Furthermore we manage to keep the $\liminf$ in the proof that strong discrete universality implies the strong continuous universality. We also prove that the discrete/continuous universality implies discrete/continuous hybrid universality.  While the converse is trivial, this has not been known previously and  the proofs have had to be  modified to take the hybrid universality part into account. This has made the proofs longer and a bit more technical.    Our proof is different than Sourmelidis proof and does not use any results from linear dynamics. Like Sourmelidis proof, it does not use the fact that the functions in questions are zeta-functions or can be represented by Dirichlet series.  Our main tools are the Mergelyan theorem, Kronecker's theorem  for the hybrid universality part, the fact that a continuous function on a compact set is uniformly continuous, and the triangle inequality. We also need to apply the universality theorems not on the original set $K$, but on a slightly enlarged set $K_0$ and  on a union  $K_1$ of suitably shifted copies of this set.

\section{Zero-free universality}

In this paper we  let $\operatorname{meas}(A)$ denote the Lebesgue measure of the set $A \subset \R$. We let
\begin{gather*}
  \strip =\{s \in \C :\sigma_1 <\Re(s)<\sigma_2 \}
\end{gather*}
denote a vertical strip in the complex plane. In many of the applications, such as for the Riemann zeta-function we have $\sigma_1=1/2$ and $\sigma_2=1$. However there are cases with other strips also, so we will let $\sigma_1$ and $\sigma_2$ be undefined. We will also allow  $\sigma_1=-\infty$ or $\sigma_2=\infty$ so that $\strip$ denotes a half-plane or a plane. We do not have any really natural examples 
for the plane/halfplane-case, but the proof works. For simplicity we call it a strip throughout the paper, rather than a "strip/halfplane/plane".
\begin{defn} \label{def1}
  We say that $Z$ satisfies the continuous universality theorem in the strip  $\strip$ if for any $\varepsilon>0$, any compact set $K \subset \strip$ with connected complement, and any zero-free function $f$ on $K$ that is analytic in its interior we have that
\begin{gather*}
  \liminf_{T  \to \infty} \frac 1 T \operatorname{meas} \left \{0 \leq t  \leq T:\max_{s \in K}  \abs{Z(s+it)-f(s)}<\varepsilon \right \}>0.
 \end{gather*}
\end{defn} 
When $Z$ satisfies the continuous universality theorem in the strip $\strip$ we say that $Z$ is universal in the strip $\strip$. While the canonical example  is $Z(s)=\zeta(s)$ the Riemann zeta-function which is universal in the strip $1/2<\Re(s)<1$, this result have been proved for many different $L$-functions such as $L$-functions of the Selberg class \cite{NagSte}. For more examples see \cite{Steuding} or \cite{Matsumoto}. 
\begin{defn} \label{def2}
  We say that $Z$ satisfies the $\alpha$-discrete universality theorem in the strip  $\strip$ if for any $\varepsilon>0$, any compact set $K \subset \strip$ with connected complement, and any zero-free function $f$ on $K$ that is analytic in its interior we have that
\begin{gather*} \liminf_{N \to \infty} \frac 1 N \#  \left \{1 \leq n \leq N:\max_{s \in K}  \abs{Z(s+in\alpha)-f(s)}<\varepsilon \right \}>0.
 \end{gather*}  
\end{defn}

\begin{defn} \label{def3}
  We say that $Z$ satisfies the continuous hybrid universality theorem in the strip  $\strip$ if for any $\varepsilon>0$, any $1<\lambda_1<\cdots< \lambda_k$  such that $ \log \lambda_1,\ldots,\log \lambda_k$ are linearly independent over $\Q$ and complex numbers $a_1,\ldots,a_k$ such that $|a_j|=1$, any compact set $K \subset \strip$ with connected complement, and any zero-free function $f$ on $K$,  that is analytic in its interior we have that  
\begin{multline*} \liminf_{T \to \infty} \frac 1 T \operatorname{meas}  \left \{0 \leq t \leq T:\max_{s \in K}  \abs{Z(s+it)-f(s)}<\varepsilon,
\max_{1\leq j \leq k}|\lambda_j^{-it}-a_j|<\varepsilon \right \}>0.
 \end{multline*}  
\end{defn} 

\begin{defn} \label{def4}
  We say that $Z$ satisfies the $\alpha$-discrete hybrid universality theorem in the strip  $\strip$ if for any $\varepsilon>0$, any  $1<\lambda_1<\cdots< \lambda_k$  such that $ \log \lambda_1,\ldots,\log \lambda_k$ are linearly independent over $\Q$ and $\pi^{-1} \alpha \log \lambda_j \not \in \Q$  and complex numbers $a_1,\ldots,a_k$ such that $|a_j|=1$, any compact set $K \subset \strip$ with connected complement, and any zero-free function $f$ on $K$ which is analytic in its interior we have that  
\begin{multline*} \liminf_{N \to \infty} \frac 1 N \#  \left \{1 \leq n \leq N:\max_{s \in K}  \abs{Z(s+in\alpha)-f(s)}<\varepsilon,  
\max_{1\leq j \leq k}|\lambda_j^{-in\alpha}-a_j|<\varepsilon \right \}>0.
 \end{multline*}  
\end{defn} 
\begin{thm} \label{thm1}
  The following propositions are equivalent
  \begin{enumerate} 
    \item $Z$ satisfies the continuous universality theorem in the strip  $\strip$.
    \item $Z$ satisfies  the $\alpha$-discrete  universality theorem in the strip  $\strip$ for some $\alpha>0$.
    \item $Z$ satisfies  the $\alpha$-discrete  universality theorem in the strip  $\strip$ for all $\alpha>0$.
 \item $Z$ satisfies the continuous hybrid universality theorem in the strip  $\strip$.
  \item $Z$ satisfies  the $\alpha$-discrete hybrid universality theorem in the strip  $\strip$ for some $\alpha>0$.
    \item $Z$ satisfies  the $\alpha$-discrete hybrid universality theorem in the strip  $\strip$ for all $\alpha>0$.
    \end{enumerate}
\end{thm}
\begin{proof}
  The implications $6) \to 5), \, \, 5) \to 2), \, \, 6) \to 3), \, \, 3) \to 2)$ and  $4) \to 1)$  are immediate from the definitions. It remains to prove that $2) \to 1), \, \,  6) \to 4)$ and $1) \to 6)$. By Mergelyan's theorem, and since $f$ is zero-free on $K$, the function $\log f$ may be approximated by a polynomial $p$ on $K$ such that
\begin{gather} \label{eq1}
 \max_{s \in K} \abs{g(s)-f(s)} <\frac \varepsilon 4,  \qquad \text{where}  \qquad g(s)=e^{p(s)}.
\end{gather}
Now let $K_0$ be a compact subset with connected complement\footnote{such as a large filled in rectangle.}  such that $K \subset K_0 \subset \strip$ and   $d(K,K_0^\complement)=\delta_0>0$.
Since $g(s)=e^{p(s)}$, where $p$ is a polynomial, we have that $g$ is uniformly continuous on $K_0$. It follows that for some $0<\delta<\delta_0$ that
\begin{gather} \label{eq2}
  \max_{|\tau| \leq \delta} \max_{s \in K}  \abs{ g(s+i \tau)-g(s)} < \frac {\varepsilon} 4.
\end{gather}
We now divide the rest of the proof into the three main cases.
\vskip 3pt

\noindent {\bf $2) \to 1)$. Discrete universality implies continuous universality:}
We apply the discrete  universality theorem with some $\alpha>0$  on the zero-free function $g$ and the compact set $K_0$. We have that 
\begin{gather}
   \xi_1:=  \liminf_{N \to \infty} \frac 1 N \# S_N>0, \label{SNineqa} \\ \intertext{where} S_N :=       \left \{1 \leq n \leq N:\max_{s \in K_0} \abs{Z(s+in\alpha)-g(s)}<\frac \varepsilon 2 \right \}.  \label{SN2ab}
\end{gather} 
 Now assume that $t$ is in a $\delta$-neighborhood of $\alpha S_N$ so that for some $n \in S_N$ we have that $|t-n \alpha| \leq \delta$. Thus for $\tau=t- n \alpha$  we have that 
\begin{multline}  \label{eq3ab}
 \max_{s \in K} \abs{Z(s+it)-g(s+i \tau)} \qquad \text{ where } \qquad z=s+i\tau \\ =
 \max_{z \in K+i \tau}\abs{Z(z+in\alpha )-g(z)} \leq  \max_{z \in K_0}\abs{Z(z+in\alpha )-g(z)} <\frac{\varepsilon} 2,
\end{multline}
where the last inequality follows from $K+i\tau \subset K_0$ and \eqref{SN2ab} since $n \in S_N$. By the triangle inequality
\begin{gather*} 
  \abs{Z(s+it)-f(s)} \leq \abs{Z(s+it)-g(s+i\tau)}+\abs{g(s+i\tau)-g(s)}+\abs{g(s)-f(s)},
\end{gather*}
 by taking the maximum of $s \in K$ and applying the inequalities \eqref{eq3ab},\eqref{eq2} and \eqref{eq1}
we obtain the inequality
 \begin{gather} \label{effq3} \max_{s \in K} \abs{Z(s+it)-f(s)}<\varepsilon,\end{gather}
if $t$ is in a $\delta$-neighbourhood of $\alpha S_N$. Thus it follows from \eqref{SNineqa} with $T=\alpha N$ that
\begin{gather} \label{conc1} \liminf_{T  \to \infty} \frac 1 T \operatorname{meas} \left \{0 \leq t \leq T:\max_{s \in K}  \abs{Z(s+it)-f(s)}<\varepsilon  \right \} \geq \xi_2>0,
 \end{gather} 
 where $\xi_2= \xi_1\alpha^{-1} \min(\alpha,2\delta)$. \qed

\vskip 3pt
\noindent {\bf $6) \to 4)$. Discrete hybrid universality implies continuous hybrid universality:} 
 Let $(\lambda_1,\ldots,\lambda_k)$ and $(a_1,\ldots,a_k)$ satisfy the conditions in Definition \ref{def3}. The  functions $h_j(t)=\lambda_j^{-it}$ for $j=1,\ldots,k$ are uniformly continuous on the compact set $[0,4\pi (\log \lambda_j)^{-1}]$  and by periodicity $h_j$ are uniformly continuous on all of $\R$.  Thus it  follows from uniform continuity that for some $0<\delta_1<\delta$ then
\begin{gather} 
  \label{eq44}
  \max_{|\tau| \leq \delta_1} \max_{t \in \R}  \max_{j=1,\ldots,k} \abs{\lambda_j^{-it}-\lambda_j^{-it-i\tau}}<\frac {\varepsilon} 2. 
 \end{gather}
We now apply the discrete hybrid universality theorem with  some $\alpha>0$ such that $\pi^{-1}\alpha \log \lambda_j \not \in \Q$ for $j=1,\ldots,k$  on the zero-free function $g$ and the compact set $K_0$. We have that 
\begin{gather}
   \xi_3:=  \liminf_{N \to \infty} \frac 1 N \# U_N>0, \label{SNineq} \\ \intertext{where} U_N :=       \left \{1 \leq n \leq N:\max_{s \in K_0} \abs{Z(s+in\alpha)-g(s)}<\frac \varepsilon 2, \max_{j=1,\ldots,k} \abs{\lambda_j^{-in \alpha}-a_j}<\frac \varepsilon 2 \right \}.  \label{SN2}
\end{gather}
   Now assume that $t$ is in a $\delta_1$-neighborhood of $\alpha U_N$. It follows since $0<\delta_1 \leq \delta$ and \eqref{SN2ab}, \eqref{SN2} that $t$ is also in a $\delta$-neighborhood of $\alpha S_N$. Thus the inequality 
 $\eqref{effq3}$ holds for $t$ in a $\delta_1$-neighborhood of $\alpha U_n$. Also for  $t$ such that $|t-n\alpha| \leq \delta_1$ for $n \in U_N$ we have by the triangle inequality  and the inequalities \eqref{eq44} and \eqref{SN2} that
\begin{gather} \label{effq4}
  \abs{\lambda_j^{-it}-a_j}  \leq   \abs{\lambda_j^{-it}-\lambda_j^{-in\alpha}}+ \abs{\lambda_j^{-in\alpha}-a_j} < \frac {\varepsilon}2+ \frac \varepsilon 2=\varepsilon,  
\end{gather}
so for $t$ in a $\delta_1$-neighborhood of $\alpha U_N$ we have that \eqref{effq3} and \eqref{effq4} holds. Thus it follows from \eqref{SNineq}  with $T=\alpha N$ that
\begin{multline} \liminf_{T  \to \infty} \frac 1 T \operatorname{meas} \left \{0 \leq t \leq T:\max_{s \in K}  \abs{Z(s+it)-f(s)}<\varepsilon, \right. \\ \left.
 \max_{j=1,\ldots,k}\abs{\lambda_j^{-it}-a_j}< \varepsilon   \right \} \geq \xi_4>0, \label{conc2}
 \end{multline}  
 where $\xi_4=\xi_3\alpha^{-1}  \min(\alpha,2\delta_1)$. \qed
\vskip 3pt

\noindent {\bf $1) \to 6)$ Continuous universality implies discrete hybrid universality:}
Let $\alpha>0$ and $1<\lambda_1<\cdots<\lambda_k$, and $a_1,\ldots,a_k$ be given such that they satisfy the conditions of Definition \ref{def4}.  Let $M \geq 1$ be a large enough integer so that \begin{gather}\label{Mdef} 0<\alpha/M<\delta.\end{gather}
Let $N \geq 1$ be a sufficiently large integer so that 
\begin{gather*}K_0 \subset \{z \in \C: -N\alpha \leq \Im(z) \leq N \alpha \},
\end{gather*}
 and let
\begin{gather} \label{K1def}
 K_1:=\bigcup_{l=0}^L \bigcup_{m=1}^{M} K_{m,l} \qquad \text{where} \qquad K_{m,l}:=K_0+i(m M_1+l L_1)\alpha, \\ \intertext{where}
 \label{M1def} M_1=3N+M^{-1}, \qquad \text{where} \qquad L_1=4MN,  
\end{gather}
is an integer and where $L$ is defined to be the smallest integer such that
\begin{gather} \label{iii}
  \max_{|b_j|=1} \min_{0 \leq l \leq L} \max_{1 \leq j \leq k} \abs{\lambda_j^{-il L_1 \alpha}- b_j}<\varepsilon
\end{gather}
holds. That $L$ exists follows from Kronecker's theorem by the following argument: If there exist no finite $L$, then there exist some $k$-tuple $(b_1,\ldots,b_k)$ with $|b_1|=\cdots=|b_k|=1$ such that the approximation
\begin{gather*}
 \max_{1 \leq j \leq k} \abs{\lambda_j^{-il L_1 \alpha}- b_j}<\frac{\varepsilon} 2 
\end{gather*}
is not possible for any integer $l$. Since $\pi^{-1} L_1 \alpha \log \lambda_j$ are all irrational and linearly independent over $\Q$ this would give a contradiction to Kronecker's theorem. We notice that by the way the constants $M_1$ and $L_1$ have been chosen,  the compact sets $K_{m,l} \subset \strip$  are disjoint, and thus $K_1 \subset \strip$ is a compact set with connected complement.
 We  now define the function $h$ on $K_1$ by \begin{gather} \label{oj} h(s+ i(m M_1+l L_1)\alpha):=g(s), \qquad s \in K_0, \end{gather} for $m=1,\ldots,M$ and $l=0,\ldots,L$. It is clear that $h$ is continuous and zero-free on $K_1$ and analytic in its interior. We may thus apply the continuous universality theorem for $Z$ on the function $h$ and the compact set $K_1$. We get that 
\begin{gather}\xi_5:= \liminf_{T  \to \infty} \frac 1 T \operatorname{meas} V_T >0, \label{ir} \\ \intertext{where} \label{eq5w}
V_T= \left \{0 \leq t \leq T:\max_{s \in K_1}  \abs{Z(s+it)-h(s)}<\frac \varepsilon 2 \right \}.
 \end{gather}
We now let $t \in V_T$. By \eqref{Mdef} and \eqref{M1def} there exist some $1 \leq m_0 \leq M$ and some integer $ n_0$  such that 
 \begin{gather} \tau:=n_0 \alpha- t-m_0 M_1 \alpha, \qquad \text{and} \qquad  \abs{\tau} \leq \delta. \label{ioii}
\end{gather} We can rewrite this as 
\begin{gather} n_l \alpha=t + \tau+ (m_0M_1+lL_1) \alpha, \qquad n_l =n_0+lL_1. \label{rera2}
\end{gather}
By $|\tau| \leq \delta<\delta_0$ it follows that $K+i\tau \subset K_0$ and by \eqref{K1def} we have
\begin{gather}
  K+i\tau+i(m_0M_1+l L_1) \alpha \subset K_1. \label{yuy}
\end{gather}
It follows by  \eqref{oj}, \eqref{rera2} and using the substitution $z=s+i\tau+i(m_0M_1+lL_1) \alpha$, \eqref{yuy}, and \eqref{eq5w} since $t\in V_T$ that
\begin{gather} 
\begin{split}  
\max_{s \in K}  \abs{Z(s+in_l \alpha )- g(s+i\tau)} \hskip -120pt& \\ 
&=\max_{s \in K} \abs{Z(s+it+i\tau+i(m_0M_1+l L_1)\alpha)- h(s+i\tau+i(m_0M_1+l L_1)\alpha)} \\  &= \max_{z \in K+i\tau+i(m_0M_1+l L_1)\alpha} \abs{Z(z+it)- h(z)} \\ &\leq \max_{z \in K_1} \abs{Z(z+it) - h(z)} <\frac {\varepsilon} 2,
\end{split} \label{rera7}
\end{gather}
for each $l=0,\ldots,L$. By the triangle inequality we have for $s \in K$ that
\begin{gather*}   
  \abs{Z(s+in_l \alpha )- f(s)} \leq \abs{Z(s+i n_l \alpha)- g(s+i\tau )} + \abs{g(s+i\tau)-g(s)}+\abs{g(s)-f(s)},
\end{gather*} 
and by taking the maximum over $s \in K$ we get by \eqref{rera7},\eqref{ioii}, \eqref{eq2} and \eqref{eq1}  that
\begin{gather*}
 \max_{s \in K} \abs{Z(s+i n_l \alpha)- f(s)} <\varepsilon,
\end{gather*}
for each $l=0,\ldots, L$. From \eqref{iii}  it follows that for some integer $0 \leq l_0 \leq L$ we have the inequality
\begin{gather*}
  \max_{1 \leq j \leq k} \abs{\lambda_j^{-i(n_0+L_1 l_0)\alpha}-a_j} 
  =  \max_{1 \leq j \leq k} \abs{\lambda_j^{-i l
_0 L_1 \alpha}-a_je^{in_0 \alpha}} < \varepsilon.
\end{gather*}
Thus given $t \in V_T$ there exists
$n:=n_{l_0}=n_0+l_0 L_1$ which by  \eqref{ioii} and \eqref{rera2} satisfies  
\begin{gather} \label {ij1}
  t-\delta \leq n\alpha  \leq t+C, \\ \intertext{where $C=\delta+(M^2+LL_1) \alpha$ such that}
  \max_{s \in K} \abs{Z(s+in \alpha)-f(s)}<\varepsilon, \qquad \text{and} \qquad \max_{1 \leq j \leq k} \abs{\lambda_j^{- i n \alpha}-a_j}<\varepsilon. \label{ij2}
\end{gather}
Define\footnote{This is the same as the  definition of $U_N$ but with $\varepsilon/2$ replaced by $\varepsilon$.}
\begin{gather}
W_N :=       \left \{0 \leq n \leq N:\max_{s \in K_0} \abs{Z(s+in\alpha)-g(s)}< \varepsilon , \max_{j=1,\ldots,k} \abs{\lambda_j^{-in \alpha}-a_j}<\varepsilon  \right \},  \label{SN22}
\end{gather}
 recall that $V_T$ is defined by   \eqref{eq5w}, and let  $j=0,\dots,  N   $ denote integers. It follows by \eqref{ij1} and \eqref{ij2} that
each $t  \in V_{\alpha(j+1)} \setminus V_{\alpha j}$ gives an element $n \in W_{j+C+1} \setminus W_{j-\delta}$ and  thus since $V_{\alpha(j+1)} \setminus V_{\alpha j} \subset [\alpha j,\alpha(j+1)] $ and $\operatorname{meas} (V_{\alpha(j+1)} \setminus V_{\alpha j}) \leq \alpha$ we get the inequality
\begin{gather} \# ( W_{j+C+1} \setminus  W_{j-\delta}) \geq \alpha^{-1} \operatorname{meas}  (V_{\alpha(j+1)} \setminus V_{\alpha j}). \label{yrr} \end{gather}
Since each $n$ is contained in $W_{j+C+1} \setminus W_{j-\delta}  \subset  (j-\delta,j+C+1]$     for at most $C+1+\delta$ values of $j$  it follows by summing \eqref{yrr} over $j=0,\ldots, N   $ that
\begin{gather} \label{urr2}
 (C+1+\delta) \# W_{N+C+1} \geq \alpha^{-1} \operatorname{meas}  V_{\alpha N},
\end{gather}
and it follows by dividing both sides of \eqref{urr2} with $(C+1+\delta)N$ and using  \eqref{ir} that 
\begin{gather*} \liminf_{N \to \infty} \frac 1 N \#W_N \geq  \xi_6>0, \qquad \text{where} \qquad \xi_6=\xi_5/(C+1+\delta),
 \end{gather*}  
which by the definition \eqref{SN22} of $W_N$ concludes our proof.
\end{proof}

\section{Strong universality}
If we remove the zero-free conditions in  the theorems of universality we arrive at the concept of strong universality. The canonical example for this case is the Hurwitz zeta-function $Z(s)=\zeta(s,\beta)$, which is known to satisfy the strong universality theorem in the strip $1/2<\Re(s)<1$, when $\beta$ is transcendental or rational and  $0<\beta<1$, $\beta \neq 1/2$. 
We obtain a result corresponding to Theorem \ref{thm1}.
\begin{thm} \label{thm2} The following propositions are equivalent
  \begin{enumerate} 
    \item $Z$ satisfies the strong continuous universality theorem in the strip  $\strip$.
    \item $Z$ satisfies  the strong $\alpha$-discrete  universality theorem in the strip  $\strip$ for some $\alpha>0$.
    \item $Z$ satisfies  the strong $\alpha$-discrete  universality theorem in the strip  $\strip$ for all $\alpha>0$.
 \item $Z$ satisfies the strong continuous hybrid universality theorem in the strip  $\strip$.
  \item $Z$ satisfies  the strong $\alpha$-discrete hybrid universality theorem in the strip  $\strip$ for some $\alpha>0$.
    \item $Z$ satisfies  the strong $\alpha$-discrete hybrid universality theorem in the strip  $\strip$ for all $\alpha>0$.
    \end{enumerate}
\end{thm}
\begin{proof}
 The proof is almost word by word the same as the proof of Theorem \ref{thm1}. 
The main difference is that we take $g(s)=p(s)$, where 
$p$ is a polynomial instead of $g(s)=e^{p(s)}$ in \eqref{eq1}. 
The "zero-free"  in the proofs should be removed and the respective strong universality theorems should be used rather than the corresponding zero-free universality theorems.
\end{proof}

\section{Mixed and joint universality}

When $Z_1(s),\ldots,Z_n(s)$ each satisfies a universality theorem
we may ask if the same imaginary shift works simultaneously. We will allow that $Z_k$ for $k=1,\ldots,i$ satisfies the zero-free universality theorem and that $Z_k$ for $k=i+1,\ldots,i+j$ satisfies the strong universality theorem. This is usually what in the literature \cite{RomMat} is called mixed  universality or mixed joint universality. 
 We will call it $(i,j)$-mixed universality.  When  $j=0$ we talk about joint universality and when $i=0$ we talk about strong joint universality.  We give the following analogs of Definition \ref{def1}-Definition \ref{def4} for mixed joint universality theorems. We will use the following  convention in $\C^n$
\begin{gather*}
 \abs{z}=\max_{1 \leq j \leq n}\abs{z_j}, \qquad \text{when} \qquad z=(z_1,\ldots,z_n) 
\end{gather*}
so the absolute value will denote the sup-norm\footnote{Any norm on $\C^n$ could be used here as well.}.

\begin{defn} \label{def6}
  We say that $Z(s)=(Z_1(s),\ldots,Z_{i+j}(s))$ satisfies the $(i,j)$-mixed  continuous universality theorem in the strip  $\strip$ if for any $\varepsilon>0$, compact set $K \subset \strip$ with connected complement, and any  continuous  functions $f_1,\ldots,f_{i+j}$ on $K$ that are analytic in its interior, and such that $f_1,\ldots f_i$ are zero-free on $K$ we have with $f(s)=(f_1(s),\ldots,f_{i+j}(s))$  that
\begin{gather*} \liminf_{T  \to \infty} \frac 1 T \operatorname{meas} \left \{0 \leq t \leq T:\max_{s \in K}  \abs{Z(s+it)-f(s)}<\varepsilon \right \}>0.
 \end{gather*}
\end{defn} 
\noindent We  say that $Z_1(s),\ldots,Z_i(s)$ are jointly universal in the strip $\strip$ if $Z(s)=(Z_1(s),\ldots,Z_i(s))$ satisfies the $(i,0)$-mixed continuous universality theorem in the strip $\strip$. The canonical example  is joint universality of Dirichlet $L$-functions, where $L(s,\chi_1),\ldots,L(s,\chi_i)$ are jointly universal in the strip $\halv <\Re(s)<1$ when $\chi_1,\ldots,\chi_i$ are pairwise nonequivalent Dirichlet characters. This was proved by Voronin \cite{Voronin2}, Gonek \cite{Gonek} and Bagchi \cite{Bagchi,Bagchi2} independently. This result have been extended to the Selberg class for $d \geq 2$ \cite{Lee}, when the $L$-functions involved satisfy some natural conditions. Mixed universality was first considered for the pair $Z(s)=(\zeta(s),\zeta(s,\beta))$ for $0<\beta<1$ where $\beta$ is transcendental in \cite{Mishou1}  and where $\beta \neq 1/2$ is rational in \cite{SanSte}. For more about the history of mixed (joint) universality, see \cite{RomMat}.
\begin{defn} \label{def7}
   We say that $Z(s)=(Z_1(s),\ldots,Z_{i+j}(s))$ satisfies the $(i,j)$-mixed  $\alpha$-discrete universality theorem in the strip  $\strip$ if for any $\varepsilon>0$, compact set $K \subset \strip$ with connected complement,  and any  continuous  functions $f_1,\ldots,f_{i+j}$ on $K$ that are analytic in its interior, and such that $f_1,\ldots f_i$ are zero-free on $K$
   we have with $f(s)=(f_1(s),\ldots,f_{i+j}(s))$  that
\begin{gather*} \liminf_{N \to \infty} \frac 1 N \#  \left \{1 \leq n \leq N:\max_{s \in K}  \abs{Z(s+in\alpha)-f(s)}<\varepsilon \right \}>0.
 \end{gather*}  
\end{defn}

\begin{defn} \label{def8}
  We say that $Z(s)=(Z_1(s),\ldots,Z_{i+j}(s))$ satisfies the $(i,j)$-mixed continuous hybrid universality theorem in the strip  $\strip$ if for any $\varepsilon>0$, any $1<\lambda_1<\cdots< \lambda_k$  such that $ \log \lambda_1,\ldots,\log \lambda_k$ are linearly independent over $\Q$ and complex numbers $a_1,\ldots,a_k$ such that $|a_j|=1$, any compact set $K \subset \strip$ with connected complement,  and any  continuous  functions $f_1,\ldots,f_{i+j}$ on $K$ that are analytic in its interior, and such that $f_1,\ldots f_i$ are zero-free on $K$
 we have with $f(s)=(f_1(s),\ldots,f_{i+j}(s))$   that  
\begin{multline*} \liminf_{T \to \infty} \frac 1 T \operatorname{meas}  \left \{0 \leq t \leq T:\max_{s \in K}  \abs{Z(s+it)-f(s)}<\varepsilon, 
\max_{1\leq j \leq k}|\lambda_j^{-it}-a_j|<\varepsilon \right \}>0.
 \end{multline*}  
\end{defn} 

\begin{defn} \label{def9}
  We say that $Z(s)=(Z_1(s),\ldots,Z_{i+j}(s))$ satisfies the $(i,j)$-mixed    $\alpha$-discrete hybrid universality theorem in the strip  $\strip$ if for any $\varepsilon>0$, any  $1<\lambda_1<\cdots< \lambda_k$  such that $ \log \lambda_1,\ldots,\log \lambda_k$ are linearly independent over $\Q$ and $\pi^{-1} \alpha \log \lambda_j \not \in \Q$  and complex numbers $a_1,\ldots,a_k$ such that $|a_j|=1$, any compact set $K \subset \strip$ with connected complement, 
   and any  continuous  functions $f_1,\ldots,f_{i+j}$ on $K$ that are analytic in its interior, and such that $f_1,\ldots f_i$ are zero-free on $K$
  we have with $f(s)=(f_1(s),\ldots,f_{i+j}(s))$   that  
  \begin{multline*} \liminf_{N \to \infty} \frac 1 N \#  \left \{1 \leq n \leq N:\max_{s \in K}  \abs{Z(s+in\alpha)-f(s)}<\varepsilon,  
\max_{1\leq j \leq k}|\lambda_j^{-in\alpha}-a_j|<\varepsilon \right \}>0.
 \end{multline*}  
\end{defn} 

\begin{thm} \label{thm3} The following propositions are equivalent
  \begin{enumerate} 
    \item $Z(s)$ satisfies the $(i,j)$-mixed   continuous universality theorem in the strip  $\strip$.
    \item $Z(s)$ satisfies   the $(i,j)$-mixed   $\alpha$-discrete  universality theorem in the strip  $\strip$ for some $\alpha>0$.
    \item $Z(s)$ satisfies  the $(i,j)$-mixed   $\alpha$-discrete  universality theorem in the strip  $\strip$ for all $\alpha>0$.
 \item $Z(s)$ satisfies the $(i,j)$-mixed  continuous hybrid universality theorem in the strip  $\strip$.
  \item $Z(s)$ satisfies  the  $(i,j)$-mixed  $\alpha$-discrete hybrid universality theorem in the strip  $\strip$ for some $\alpha>0$.
    \item $Z(s)$ satisfies  the $(i,j)$-mixed  $\alpha$-discrete hybrid universality theorem in the strip  $\strip$ for all $\alpha>0$.
    \end{enumerate}
\end{thm}

\begin{proof}
  The proof is almost identical to the proof of Theorem 1, as long as we interpret the absolute values as a norm  in $\C^n$, where $n=i+j$. The changes we need to do are the following: In \eqref{eq1} we  use
\begin{gather} g(s)=(e^{p_1(s)},\ldots,e^{p_i(s)},p_{i+1}(s),\ldots,p_{i+j}(s)) \label{gdef2}
\end{gather}  instead of $g(s)=e^{p(s)}$, 
where by Mergeyan's theorem $p_1(s),\ldots,p_i(s)$ are polynomal approximations of $\log f_1(s),\ldots,\log f_i(s)$ and $p_{i+1}(s),\ldots, p_{i+j}(s)$ are polynomal approximations  of $f_{i+1}(s),\ldots,f_{i+j}(s)$. Then when we apply the unversality theorems, we need to apply the $(i,j)$-mixed versions of the theorems.
\end{proof}
 Hybrid universality have a lot of applications. For example it is a consequence of  
\cite[Theorem 2]{NakPan} and Theorem \ref{thm3}  that
\begin{thm}
Suppose that $Z_1(s),\ldots,Z_{n}(s)$ are jointly universal in the strip  $\strip$. 
 Then any linear combination 
$$
  A(s)=\sum_{k=1}^{n} D_k(s) Z_{k}(s),
$$
where $D_k(s)$ are absolutely convergent Dirichlet series in the strip $\strip$, at least two which are not identically zero, is strongly universal in the strip $\strip$.
\end{thm}
The difference compared to \cite[Theorem 2]{NakPan} is that we have replaced the hybrid joint universality condition   with a joint universality condition.
Similarly in \cite[Theorem 1, Corollary 1, Theorem 3]{NakPan}, the assumption that $L_1,\ldots,L_n$ are hybridly jointly universal can be relaxed to jointly universal. As a consequence of the simplest case of \cite[Theorem 1]{NakPan} and Theorem \ref{thm3} we have that whenever $Z(s)$ is universal in the strip $\strip$ and $A(s)$ is an absolutely convergent zero-free Dirichlet series in $\strip$, then $A(s)Z(s)$ is also universal in the strip $\strip$. Simple cases that follows from the universality of the Riemann zeta-function are $\zeta(s)/\zeta(2s)$ and $\zeta(s)/\zeta(s+1)$ which are universal in the strip $1/2<\Re(s)<1$.

\section{Universality in short intervals}

Recently Antanas  Laurin\v cikas \cite{Laurincikas} proved that in the Voronin universality theorem for the Riemann zeta-function, we may choose $t$ in short intervals $[T,T+\omega(T)]$, when $T^{1/3} (\log T)^{26/15} \leq \omega(T) \leq T$. Shortly thereafter \cite{Laurincikas2} he proved the corresponding discrete universality theorem. We may extend our results that relates discrete, continuous and hybrid universality to the short interval case such that the result in \cite{Laurincikas2} follows from Theorem \ref{thm4} and the result in \cite{Laurincikas}. In the below definitions we  assume that $\omega:\R^+ \to \R^+$. In practice we are interested to find $\omega(T)$ that grows to infinity as slow as possible. We choose to state our definitions in this case in their $(i,j)$-mixed form as their joint,  strong and zero-free versions are consequences of this.

\begin{defn} \label{def10}
  We say that $Z(s)=(Z_1(s),\ldots,Z_{i+j}(s))$ satisfies the $(i,j)$-mixed  continuous universality theorem in the strip  $\strip$ and  short intervals $[T,T+\omega(T)]$ if for any $\varepsilon>0$, compact set $K \subset \strip$ with connected complement, and any  continuous  functions $f_1,\ldots,f_{i+j}$ on $K$ that are analytic in its interior, and such that $f_1,\ldots f_i$ are zero-free on $K$ we have with $f(s)=(f_1(s),\ldots,f_{i+j}(s))$  that
\begin{gather*} \liminf_{T  \to \infty} \frac 1 {\omega(T)} \operatorname{meas} \left \{T \leq t \leq T+\omega(T):\max_{s \in K}  \abs{Z(s+it)-f(s)}<\varepsilon \right \}>0.
 \end{gather*}
\end{defn} 
\begin{defn} \label{def11}
   We say that $Z(s)=(Z_1(s),\ldots,Z_{i+j}(s))$ satisfies the $(i,j)$-mixed  $\alpha$-discrete universality theorem in the strip  $\strip$ and  short intervals $[T,T+\omega(T)]$  if for any $\varepsilon>0$, compact set $K \subset \strip$ with connected complement,  and any  continuous  functions $f_1,\ldots,f_{i+j}$ on $K$ that are analytic in its interior, and such that $f_1,\ldots f_i$ are zero-free on $K$
   we have with $f(s)=(f_1(s),\ldots,f_{i+j}(s))$  that
\begin{gather*} \liminf_{N \to \infty}  \frac 1 {\omega(N)} \#  \left \{N \leq n \leq N+\omega(N) :\max_{s \in K}  \abs{Z(s+in\alpha)-f(s)}<\varepsilon \right \}>0.
 \end{gather*}  
\end{defn}

\begin{defn} \label{def12}
  We say that $Z(s)=(Z_1(s),\ldots,Z_{i+j}(s))$ satisfies the $(i,j)$-mixed continuous hybrid universality theorem in the strip  $\strip$  and  short intervals $[T,T+\omega(T)]$  if for any $\varepsilon>0$, any $1<\lambda_1<\cdots< \lambda_k$  such that $ \log \lambda_1,\ldots,\log \lambda_k$ are linearly independent over $\Q$ and complex numbers $a_1,\ldots,a_k$ such that $|a_j|=1$, any compact set $K \subset \strip$ with connected complement,  and any  continuous  functions $f_1,\ldots,f_{i+j}$ on $K$ that are analytic in its interior, and such that $f_1,\ldots f_i$ are zero-free on $K$
 we have with $f(s)=(f_1(s),\ldots,f_{i+j}(s))$   that  
\begin{multline*} \liminf_{T \to \infty} \frac 1 {\omega(T)} \operatorname{meas}  \left \{T \leq t \leq T+\omega(T):\max_{s \in K}  \abs{Z(s+it)-f(s)}<\varepsilon, \right. \\ \,
 \left. \max_{1\leq j \leq k}|\lambda_j^{-it}-a_j|<\varepsilon \right \}>0.
 \end{multline*}  
\end{defn} 

\begin{defn} \label{def13}
  We say that $Z(s)=(Z_1(s),\ldots,Z_{i+j}(s))$ satisfies the $(i,j)$-mixed    $\alpha$-discrete hybrid universality theorem in the strip  $\strip$  and  short intervals $[N,N+\omega(N)]$  if for any $\varepsilon>0$, any  $1<\lambda_1<\cdots< \lambda_k$  such that $ \log \lambda_1,\ldots,\log \lambda_k$ are linearly independent over $\Q$ and $\pi^{-1} \alpha \log \lambda_j \not \in \Q$  and complex numbers $a_1,\ldots,a_k$ such that $|a_j|=1$, any compact set $K \subset \strip$ with connected complement, 
   and any  continuous  functions $f_1,\ldots,f_{i+j}$ on $K$ that are analytic in its interior, and such that $f_1,\ldots f_i$ are zero-free on $K$
  we have with $f(s)=(f_1(s),\ldots,f_{i+j}(s))$   that  
  \begin{multline*} \liminf_{N \to \infty} \frac 1 {\omega(N)} \#  \left \{N \leq n \leq N+\omega(N):\max_{s \in K}  \abs{Z(s+in\alpha)-f(s)}<\varepsilon, \right. \\  \, \left. \max_{1\leq j \leq k}|\lambda_j^{-in\alpha}-a_j|<\varepsilon \right \}>0.
 \end{multline*}  
\end{defn} 

Our most general result is the following Theorem which with $\omega(T)=T$ implies Theorem \ref{thm3}, which implies Theorem \ref{thm1} and Theorem \ref{thm2}.

\begin{thm} \label{thm4} Let $\omega:\R^+ \to \R^+$ be a function such that $\lim_{t \to \infty} \omega(t)=\infty$ and \begin{gather} \label{iiij} \inf_{t \geq 1} \frac{\omega(\alpha t)}{\omega(t)}>0, \end{gather}
 when $\alpha>0$. Then the following propositions are equivalent
  \begin{enumerate} 
    \item $Z(s)$ satisfies the $(i,j)$-mixed   continuous universality theorem in the strip  $\strip$ and short intervals $[T,T+\omega(T)]$.
    \item $Z(s)$ satisfies   the $(i,j)$-mixed   $\alpha$-discrete  universality theorem in the strip  $\strip$ and short intervals $[N,N+\omega(N)]$ for some $\alpha>0$.
    \item $Z(s)$ satisfies  the $(i,j)$-mixed   $\alpha$-discrete  universality theorem in the strip  $\strip$ and short intervals $[N,N+\omega(N)]$ for all $\alpha>0$.
 \item $Z(s)$ satisfies the $(i,j)$-mixed  continuous hybrid universality theorem in the strip  $\strip$ and short intervals $[T,T+\omega(T)]$.
  \item $Z(s)$ satisfies  the  $(i,j)$-mixed  $\alpha$-discrete hybrid universality theorem in the strip  $\strip$ and short intervals $[N,N+\omega(N)]$ for some $\alpha>0$.
    \item $Z(s)$ satisfies  the $(i,j)$-mixed  $\alpha$-discrete hybrid universality theorem in the strip  $\strip$ and short intervals $[N,N+\omega(N)]$ for all $\alpha>0$.
    \end{enumerate}
\end{thm}

\begin{proof} Like in the proof of Theorem \ref{thm3}, the proof is very similar to the proof of Theorem \ref{thm1}, as long as we interpret the absolute value as a norm in $\C^n$ and we define $g(s)$ by \eqref{gdef2} instead as in \eqref{eq1}.
We also need to use the corresponding universality theorems in their $(i,j)$-mixed form in short intervals. Thus the quantities $S_N, U_N, V_T$ defined by  \eqref{SN2ab},\eqref{SN2},\eqref{eq5w}  must instead be defined by  
\begin{align*}
S_N &:=      \left \{N \leq n \leq N+\omega(N):\max_{s \in K_0} \abs{Z(s+in\alpha)-g(s)}<\frac \varepsilon 2 \right \}, \\  U_N &:=      \begin{multlined}[t] \left \{N \leq n \leq N+\omega(N):\max_{s \in K_0} \abs{Z(s+in\alpha)-g(s)}<\frac \varepsilon 2, \right. \\ \left. \max_{j=1,\ldots,k} \abs{\lambda_j^{-in \alpha}-a_j}<\frac \varepsilon 2 \right \}, \end{multlined}
\\ 
V_T&:= \left \{T \leq t \leq T+\omega(T):\max_{s \in K_1}  \abs{Z(s+it)-h(s)}<\frac \varepsilon 2 \right \}. 
\\ \intertext{Then \eqref{SNineqa},  \eqref{SNineq}, and \eqref{ir} must be replaced by} 
\xi_1&:=  \liminf_{N \to \infty} \frac 1 {\omega(N)} \# S_N>0,  \qquad \xi_3:=  \liminf_{N \to \infty} \frac 1 {\omega(N)} \# U_N>0, \\ \intertext{and}    \xi_5&:= \liminf_{T  \to \infty} \frac 1 {\omega(T)} \operatorname{meas} V_T >0.
\end{align*}
For the conclusions \eqref{conc1} and \eqref{conc2} we must use
the condition $\inf_{N \geq 1}\omega(N)/ \omega(\alpha N)>0$ which follows by replacing $\alpha$ with $\alpha^{-1}$ in \eqref{iiij} .
 Finally at the end of page \pageref{urr2} we must sum over $j=N  \ldots,\lfloor N+\omega(N) \rfloor$ and instead of \eqref{urr2} we obtain
\begin{gather*}
(C+1+\delta) \# (W_{N+\omega(N)+C+1}\setminus W_{N-\delta}) \geq  \alpha^{-1}  \operatorname{meas} V_{\alpha N}.
\end{gather*}
 which  together with the condition  \eqref{iiij}  implies our result.
\end{proof}

 \begin{ack} The author is grateful to Athanasios Sourmelidis for giving useful comments on previous versions of this manuscript and for giving an interesting talk at the international conference on probability theory and number theory 2023, in Palanga, which gave me inspiration to write this paper.
 \end{ack}

\bibliographystyle{plain}

\end{document}